\documentclass[preprint,12pt]{elsarticle}

\usepackage{amssymb}
\usepackage{stmaryrd}
\usepackage{tipa}
\usepackage{amsfonts,amsmath,latexsym}
\usepackage{mathrsfs}  
\usepackage{amsbsy}
\usepackage{indentfirst}
\usepackage{amsmath}

\numberwithin{equation}{section}

\newtheorem{theorem}{\textbf{Theorem}}[section]

\newtheorem{remark}[theorem]{Remark}

\newcommand{\beqnar}{\begin{eqnarray*}}
\newcommand{\eeqnar}{\end{eqnarray*}}
\newcommand{\ba}{\begin{array}}
\newcommand{\ea}{\end{array}}

\newenvironment{proof}[1]{\begin{trivlist}\item {\it
\bf Proof.}\quad} {\qed\end{trivlist}}

\journal{}

\begin{document}

\begin{frontmatter}



\title{On Bernstein Type Inequalities  for Stochastic Integrals of Multivariate Point Processes}


\author{Hanchao Wang\footnote{Corresponding author, email: wanghanchao@sdu.edu.cn.}}

\address{ Zhongtai Security Institute for Financial Studies,  Shandong University,  Jinan,  250100, PRC
    }
  \author{Zhengyan Lin, Zhonggen Su}

\address{ School of Mathematical Sciences, Zhejiang University, Hangzhou, 310027, PRC
    }

\begin{abstract}
  We consider the stochastic integrals of multivariate point processes and study their concentration phenomena. In particular, we obtain a Bernstein type of concentration inequality through  Dol\'{e}ans-Dade exponential formula and a uniform exponential inequality using a generic chaining argument. As applications, we obtain a upper bound for a sequence of discrete time martingales indexed by a class of functionals, and so  derive the rate of  convergence for nonparametric maximum likelihood estimators, which is an improvement of earlier work of van de Geer.
\end{abstract}

\begin{keyword}
Bernstein  inequality, Dol\'{e}ans-Dade exponential formula,  Generic chaining method,  Multivariate point process.

\end{keyword}

\end{frontmatter}


\section{Introduction}

There have been a lot of research activities around phenomena of measure concentration in the past decades. The reader is referred to excellent books like Ledoux and Talagrand \cite{lt}, Ledoux \cite{l} and nice paper like Talagrand \cite{t0} for remarkable results and powerful methods. A primary purpose of the present paper is to establish a Bernstein type  exponential concentration inequality for stochastic integrals of multivariate point processes.

 For sake of statement, we will begin with a classical Bernstein inequality for sums of independent random variables. Assume that $(\Omega,  \cal{F}, \textsf{P})$ is a   probability space so large  that we can construct all random objects of interest in it. Let $\xi_{1}$, $\xi_{2}$, $\cdots$ be  a sequence of centered independent random variables with finite variance, and denote
$S_{n}=\xi_{1}+\xi_{2}+\cdots+\xi_{n}$. If there exists a certain  constant $a>0$ such that
\begin{equation}
\textsf{E}[|\xi_k|^p]\le \frac{p!a^{p-2}}{2}\textsf{E}[\xi_k^2], \quad p\ge 2\label{obc}
\end{equation}
  then
    \begin{equation}
    \textsf{P}(|S_{n}|\ge nx)\le 2\exp(-\frac{nx^{2}}{2(K+ax)})
    \label{ob2c}
    \end{equation}
  for all $x>0$ and for all $K$ satisfying $Var(S_{n})=\sum_{i=1}^{n}\textsf{E}\xi_{n}^{2}\le K$.

  (\ref{obc}) was due to  Bernstein \cite{br}, and so (\ref{obc}) is now referred as Bernstein condition. Since then  various extensions and improvement have appeared in literature, among which are Bennett \cite{b,b1}, Hoeffding \cite{h},  Freedman \cite{f},  Bentkus \cite{be1, be2},  Fan et al. \cite{fgl}. A very recent nice book is Bercu et al \cite{bdr} which gives a very clear exposition on concentration inequalities for sums of independent random variables and martingales.

An important extension of Bernstein inequality is to both discrete time martingales and continuous time martingales. In particular,  Freedman \cite{f} first obtained the  Bernstein  inequality of discrete time martingales with bounded jumps, and  then  Shorack and Wellner \cite{sw} extended Freedman's result to continuous time martingales. More precisely, let $(\Omega, \mathcal{F}, (\mathcal{F}_{t})_{t\in [0,T]},\textsf{ P})$  be a stochastic basis, $\{M_{t}\}_{t\ge 0}$ be a locally square integrable martingale with respect to the filtration $\{\mathcal{F}_{t}\}_{t \in [0,T]}$ with $M_{0}=0$. Denote the jump by $\Delta M_{t}=M_t-M_{t-}$ and  the predictable variation by $V_{t}=<M, M>_{t}$, $t>0$.  Assume that
 \begin{equation}
    |\Delta M_{t}|\le K, \quad t>0 \label{BJump}
    \end{equation}
for a positive constant $K$. Then  for each $x, y>0$,
\begin{equation}
\textsf{P}(M_{t}\ge x~\text{and}~V_{t}\le y^{2}~\text{for some}~t)\le \exp\Big(-\frac{x^{2}}{2(xK+y^{2})}\Big). \label{ShW}
\end{equation}

The bounded jump assumption (\ref{BJump}) can be relaxed. In fact, van de Geer \cite{v} improved the above result under Bernstein condition. For each $m \ge 2$, consider the process $\{\sum_{s\le t}|\Delta M_{s}|^{m}\}_{t\ge 0}$ and its predictable compensator $\{ V_{m, t}\}_{t\ge 0}$. If there exist a constant $K>0$ and a predictable process $\{R_{t}\}_{t\ge 0}$ such that
\begin{equation}
V_{m,t}\le \frac{m!}{2}K^{m-2}R_{t}, \quad t\ge 0, \label{Geer0}
\end{equation}
 then  for each $x, y>0$,
\begin{equation}
\textsf{P}(M_{t}\ge x~\text{and}~V_{t}\le y^{2}~\text{for some}~t)\le \exp\Big(-\frac{x^{2}}{2(xK+y^{2})}\Big). \label{Geer1}
\end{equation}
We remark that any locally square integrable martingale can be represented as the sum of continuous local martingale and pure jump local martingale. The nonzero  continuous local martingale part indeed played a crucial role in the proof of  both (\ref{ShW}) and (\ref{Geer1}). Now it is natural to ask what happens for a pure jump local martingale.  It is an interesting and challenging mathematical problem to establish a concentration inequality for general pure jump local martingales. We shall restrict ourselves to stochastic integrals of multivariate point processes.

Let $(E, {\cal E})$ be a Blackwell space. Assume that $\{T_k\}_{k\ge 1}$  be a sequence of strictly increasing positive  random  variables, $\{X_k\}_{k\ge 1}$ a sequence of $E$-valued random variables and $X_k$ is measurable with respect to  ${\cal F}_{T_k}$ for each $k\ge 1$. A multivariate point process  is an integer-valued random variables defined by
\begin{equation}
 \mu(dt,dx)=\sum_{k\ge 1}1_{\{T_{k}<\infty\}}\varepsilon_{(T_{k},X_{k})}(dt\times dx). \label{mp}
 \end{equation}

 We note that Poisson point process and compound Poisson point process are classic and well-studied examples of   multivariate point processes.  We shall be interested in stochastic integrals of a predictable process with respect to  the measure $\mu$. Let $\nu$ be the predictable compensator of $\mu$ and assume that $\nu$ admits the disintegration
\begin{equation}
 \nu(\omega, dt,dx)= dA_t(\omega)K_{\omega, t}(dx). \label{mp1}
 \end{equation}
 where $K$ is a transition probability from $(\Omega\times [0,T], \mathcal{P})$ in to $(E,\mathcal{E})$, $A$ is an increasing c\`{a}dl\'{a}g predictable process.  Denote $a_{t}=\nu(\{t\}\times \mathbb{R})$. It is easy to see that the process $a\equiv 0$ if $\mu$ is a L\'{e}vy point process. However, what we are more interested in the case $a\neq 0$, namely $\Delta A_t\neq 0$.

Given a predictable function $W$ on $\tilde{\Omega}$, $\tilde{\Omega}=\Omega\times \mathbb{R}_{+}\times \mathbb {R}$, define the stochastic integral
\begin{equation}
\hat{W}_{t}=\int_{\mathbb{R}}W(t,x)\nu(\{t\}\times dx). \label{StIn}
\end{equation}
In addition, put
\begin{equation}
 \Xi(W)_{t}= \max\{0, (W-\hat{W})\}*\nu_t+(1-a_s)\max\{0, -\hat{W}_t\}. \label{Zeta}
\end{equation}
and
\begin{equation}
 C(W)_t=\langle W*(\mu-\nu),W*(\mu-\nu)\rangle_t . \label{CW}
\end{equation}
An easy computation, see Chapter 2 of Jacod and Shiryaev \cite {js},  implies
\begin{equation}
 C(W)_t=(W-\hat{W})^2*\nu_t+\sum_{s\le t}(1-a_s)(\hat{W}_s)^2. \label{CW2}
\end{equation}

Motivated by (\ref{CW2}), we introduce the following quantities
\begin{equation}
Q(W,m)_t=\max\{0, (W-\hat{W})\}^m*\nu_t+\sum_{s\le t}(1-a_s)\max\{0, -\hat{W}\}^m, \quad m\ge 3. \label{QW2}
\end{equation}
The Bernstein inequality for $W*(\mu-\nu)$ reads as follows
\begin{theorem}\label{BESTIN}
 Suppose that for all $t>0$ and some $0<K<\infty$
 \begin{equation}
 \Xi(W)_{t}\le K, \quad Q(W,m)_t\le \frac{m!}{2}K^{m-2}C(W)_t, \quad m\ge 3. \label{CW3}
 \end{equation}
  Then for each $x>0,$ $y>0$,
\begin{equation}
\textsf{P}(|W*(\mu-\nu)_{t}|\ge x~\text{and}~C(W)_t\le y^{2}~\text{for some}~t)\le \exp(-\frac{x^{2}}{2(xK+y^{2})}). \label{CW4}
\end{equation}

\end{theorem}

The proof of Theorem \ref{BESTIN} will be given in Section 2. A key ingredient is Dol\'{e}ans-Dade exponential formula for semimartingales with given predictable characteristics.

Next let us turn to consider the uniform bound for a family of stochastic integrals of predictable processes with respect to multivariate point process.
Let   $(\Psi, d)$  be a metric space, $\mathcal{W}=\{W^\psi:\psi\in \Psi\}$  a family of predictable functions on $\Omega \times \mathbb{R}_{+}\times \mathbb{R}$.

Fix a $T>0$.  We denote  $X^\psi=W^\psi*(\mu-\nu)_T$ and define two metrics as follows

\begin{equation}
d_{1}(\psi_1,\psi_2)=||\Xi(W^{\psi_1}-W^{\psi_2})_T||_{\infty},
\end{equation}

\begin{equation}
d_{2}(\psi_1,\psi_2)=\sqrt{||C(W^{\psi_1}-W^{\psi_2})_T||_{\infty}}
\end{equation}
where $||\cdot||_{\infty}$ stands for norm of $L^{\infty}$.

By Theorem \ref{BESTIN}, one easily can obtain
\begin{equation}
        \textsf{P}(|X^{\psi_{1}}-X^{\psi_{2}}|>x)\le  \exp\Big(-\frac{x^{2}}{2(d_{2}^{2}(\psi_{1},\psi_{2})+d_{1}(\psi_{1},\psi_{2})x)}\Big). \label{c1}
 \end{equation}
As known to us, (\ref{c1}) is a certain increment condition. We can further derive a uniform inequality for $\sup_{\psi\in \Psi}{X^\psi}$ using a generic chaining method as in Talagrand \cite{t}. To this end, we need to introduce more notations.  For a  given metric space $(\Psi, d)$, an increasing sequence $(\mathcal{A}_n)_{n\ge 1}$ of partitions of $\Psi$ is called as admissible sequence if    $\sharp\mathcal{A}_n\le 2^{2^{n}}$. Denote by $A_n(\psi)$ the unique of element of $(\mathcal{A}_n)$ containing $\psi$, and denote by $\Upsilon_d(A_n(\psi))$ the diameter of $A_n(\psi)$ under $d$. In addition,  let
\begin{equation}
\gamma_\alpha(\Psi,d)=\inf\sup_{\psi\in \Psi} \sum_{n\ge 0}2^{n/\alpha}\Upsilon_d(A_n(\psi)), \quad \label{TalaGamma}
 \end{equation}
where the infimum is taken over all admissible sequences. We can now state a uniform inequality for $\sup_{\psi\in \Psi}{X^\psi}$ in terms of $\gamma_1$ and $\gamma_2$.
\begin{theorem}\label{mui}
Suppose that  for all $\psi\in \Psi$ and some $0< K<\infty$
 \begin{equation}
 \Xi(W^{\psi})_{T}\le K, \quad C(W^{\psi})_T^m\le \frac{m!}{2}K^{m-2}C(W^{\psi})_T, \quad m\ge 3. \label{mu2}
  \end{equation}
 Then   we have
\begin{equation}
\textsf{P} \big(\sup_{\psi\in \Psi}|X^\psi|\ge Cu (\gamma_2(\Psi,d_{2})+\gamma_1(\Psi,d_{1}))\big)\le C\exp(-\frac{u}{2}). \label{ui}
\end{equation}
Moreover, it follows
\begin{equation}
\textsf{E} \sup_{\psi\in \Psi}X^\psi\le C \big(\gamma_2(\Psi,d_{2})+\gamma_1(\Psi,d_{1})\big).\label{mr}
\end{equation}
\end{theorem}
The proof of Theorem \ref{mui} will also be given in Section 2. As applications, we will obtain a Bernstein type exponential inequality for a class of functional index empirical processes and so derive a convergence rate for nonparametric maximum likelihood estimators.  This is the content of Section 3.

\section{Proofs of Theorems 1 and 2}

{\bf Proof of Theorem \ref{BESTIN} } Clearly, it follows
    \begin{equation}
    \textsf{P}(|W*(\mu-\nu)_{t}|\ge x)\le\textsf{P}(|W|*(\mu-\nu)_{t}\ge x). \nonumber
    \end{equation}
 So  without loss of generality, we can and do assume $W>0$. For simplicity of notation,  put
\begin{equation}
X_{t}=W*(\mu-\nu)_{t}, \nonumber
 \end{equation}
 and
 \begin{equation}
 S(\lambda)_{t}=\int_{0}^{t}\int_{\mathbb{R}}(e^{\lambda (W-\hat{W})}-1-\lambda (W-\hat{W}))\nu(ds,dx)+\sum_{s\le t}(1-a_{s})(e^{-\lambda \hat{W}_{s}}-1+\lambda \hat{W}_{s}),  \nonumber
  \end{equation}
where $0<\lambda<\frac{1}{K}$. Note for any semimartingale $Y$,  the Dol\'{e}ans-Dade exponential  is
   \begin{equation}
    \mathcal{E}(Y)_{t}=e^{Y_{t}-Y_{0}-\frac{1}{2}<Y^{c},Y^{c}>_{t}}\prod_{s\le t}(1+\Delta Y_{s})e^{-\Delta Y_{s}}.\nonumber
    \end{equation}
 Since
\begin{eqnarray*}
                \Delta S(\lambda)_{t}=\frac{\int e^{\lambda W}\nu(\{t\},dx)-a_te^{\lambda \hat{W}_{t}}+(1-a_t)(1-e^{\lambda \hat{W}_{t}})}{e^{\lambda \hat{W}_{t}}}\nonumber
 \end{eqnarray*}
             and
 \begin{equation}
 \int e^{\lambda W}\nu(\{t\},dx)+1-a_t>0, \nonumber
  \end{equation}
         we can obtain
    \begin{equation}
    \Delta S(\lambda)_{t} >-1\nonumber
     \end{equation}
for all $t>0$.

We shall first show the the process  $\Big(e^{\lambda X}/\mathcal{E}(S(\lambda))\Big)_{t\ge 0}$ is a local martingale. For $X$, the jump part of $X$ is
   \begin{equation}
   \Delta X=(W-\hat{W})1_{D}-\hat{W}1_{D^{c}},
    \end{equation}
where $D$ is the thin set, which is  exhausted by $\{T_n\}_{n\ge 1}$.

We denote  by $\mu^X$ the jump measure of $X$. Let $\nu^X$ be the predictable compensator of $\mu^X$, and
 \begin{equation}
 x*\nu^{X}_{t}=(W-\hat{W})*\nu_t-\sum_{s\le t}\hat{W}_{s}(1-a_{s}).
 \end{equation}
The It\^{o} formula yields
        \begin{eqnarray*}
        e^{\lambda X_{t}}&=&1+\lambda e^{\lambda X_{t-}}\cdot X+\sum_{s\le t}(e^{\lambda X_{s}}-e^{\lambda X_{s-}}-\lambda e^{\lambda X_{s-}}\Delta X_{s} )\\
                          &=&1+\lambda e^{\lambda X_{t-}}\cdot X+e^{\lambda X_{-}}\cdot (e^{\lambda x}-1-\lambda x)*\mu^{X}.
        \end{eqnarray*}
Furthermore,
        \begin{equation}
        S(\lambda)= (e^{\lambda x}-1-\lambda x)*\nu^{X}.
        \end{equation}
     We obtain that
     \begin{equation}
     e^{\lambda X}-e^{\lambda X_{-}}\cdot S(\lambda)
      \end{equation}
     is a local martingale. Set $H=e^{\lambda X}$, $G=\mathcal{E}(S(\lambda))$, $A=S(\lambda)$ and $f(h,g)=\frac{h}{g}.$  The  It\^{o} formula yields
             \begin{eqnarray*}
             f(H,G)&=&1+\frac{1}{G_{-}}\cdot H-\frac{H_{-}}{G^{2}_{-}}\cdot G\\
             & &+\sum_{s\le \cdot}\Big(\Delta f(H,G)_{s}-\frac{\Delta H_{s}}{G_{s-}}+\frac{f(H,G)_{s-}}{G_{s-}}\Delta G_{s}\Big).
             \end{eqnarray*}
Let $N_1=H-H_{-}\cdot A$, and note $N_1$  is also a local martingale. We have
     \begin{equation}
     \frac{1}{G_{-}}\cdot H=\frac{1}{G_{-}}\cdot N_1+\frac{H_{-}}{G_{-}}\cdot  S(\lambda).
     \end{equation}
  By the definition of $\mathcal{E}(A)$, we have $G=1+G_{-}\cdot A$, thus
     \begin{equation}
     \frac{H_{-}}{G^{2}_{-}}\cdot G=\frac{H_{-}}{G_{-}}\cdot  A.
     \end{equation}
  Then
      \begin{equation}
      \frac{1}{G_{-}}\cdot H-\frac{H_{-}}{G^{2}_{-}}\cdot G=\frac{1}{G_{-}}\cdot N_1.
      \end{equation}
Noting that  $\Delta G=G_{-}\Delta A$, $\Delta N_{1}=\Delta H-H_{-}\Delta A$, we have
          \begin{eqnarray*}
           & &\Delta f(H,G)_{s}-\frac{\Delta H_{s}}{G_{s-}}+\frac{f(H,G)_{s-}}{G_{s-}}\Delta G_{s}\\
           &=&\frac{H_{s}}{G_{s}}-\frac{H_{s-}}{G_{s-}}-\frac{\Delta H_{s}}{G_{s-}}+\frac{f(H,G)_{s-}}{G_{s-}}\Delta G_{s}\\
            &=&\frac{H_{s}}{G_{s-}(1+\Delta A_{s})}-\frac{H_{s-}}{G_{s-}}-\frac{\Delta H_{s}}{G_{s-}}+\frac{f(H,G)_{s-}}{G_{s-}}\Delta G_{s}\\
            &=&\frac{H_{s-}(\Delta A_{s})(1+\Delta A_{s})-H_{s}\Delta A_{s}}{G_{s-}(1+\Delta A_{s})}\\
            &=&-\frac{\Delta N_{1s}\Delta A_{s}}{G_{s-}(1+\Delta A_{s})},
               \end{eqnarray*}
where $A$ is a predictable process, and $N$ is a local martingale. By the property of  the Stieltjes integral, we have
     \begin{equation}
     \sum_{s\le \cdot }\Delta f(H,G)_{s}-\frac{\Delta H_{s}}{G_{s-}}+\frac{f(H,G)_{s-}}{G_{s-}}\Delta G_{s}= -\frac{\Delta A}{G_{-}(1+\Delta A)} \cdot N_1.
         \end{equation}
Thus $\Big(e^{\lambda X}/\mathcal{E}(S(\lambda))\Big)_t$ is a local martingale.

Since $e^{x}\ge x+1$ and $e^{S(\lambda)_{t}}\ge \mathcal{E}(S(\lambda)_{t})$,
     \begin{equation}
     \textsf{E} \frac{e^{\lambda X_{t}}}{e^{S(\lambda)_{t}}}\le \textsf{E} \frac{e^{\lambda X_{t}}}{\mathcal{E}(S(\lambda)_{t})}=1.
      \end{equation}
Thus,
       \begin{eqnarray*}
       S(\lambda)_{t}&\le&\sum_{m=2}^{\infty}\frac{\lambda^{m}}{m!}C(W)^{m}_{t}\\
                     &\le &  \frac{\lambda^{2}}{2(1-\lambda K)}C(W)_{t}.
       \end{eqnarray*}
Set
  \begin{equation}
  A=\{X_{t}\ge x~\text{and}~C(W)_t\le y^{2}~\text{for some}~t\},
  \end{equation}
\begin{equation}
\sigma=\inf\{t:X_t\ge x\},
\end{equation}
     we have
\begin{equation} \int _A \frac{e^{\lambda X_{\sigma}}}{e^{S(\lambda)_{\sigma}}}d \textsf{P}\le 1,
 \end{equation}
 On $A$,
 \begin{equation}
 \frac{e^{\lambda X_{\sigma}}}{e^{S(\lambda)_{\sigma}}}\ge  \exp \Big(\lambda x - \frac{\lambda^{2}y^{2}}{2(1-\lambda K)}\Big),
  \end{equation}
         then
   \begin{equation}
    \textsf{P}(A)\le \exp \Big(-\lambda x+ \frac{\lambda^{2}y^{2}}{2(1-\lambda K)}\Big).
   \end{equation}
  Take $\lambda=\frac{x}{y+Kx}$, we obtain
  \begin{eqnarray*}
   \textsf{P}(X_{t}\ge x~\text{and}~C(W)_t\le y^{2}~\text{for some}~t)\le  \exp \Big(-\frac{x^{2}}{2(y^{2}+Kx)}\Big).
 \end{eqnarray*}

{\bf Proof of Theorem \ref{mui}}

By Theorem \ref{BESTIN}, we can obtain
             \begin{equation*}
        \textsf{P}(|X^{\psi_{1}}-X^{\psi_{2}}|>u)\le  \exp\Big(-\frac{u^{2}}{2(d_{2}^{2}(\psi_{1},\psi_{2})+d_{1}(\psi_{1},\psi_{2})u)}\Big).
      \end{equation*}
     and
     \begin{equation*}
        \textsf{P}\big(|X^{\psi_{1}}-X^{\psi_{2}}|>ud_{1}(\psi_{1},\psi_{2})+\sqrt{u}d_{2}(\psi_{1},\psi_{2})\big)\le  \exp(-\frac{u}{2}).
      \end{equation*}
for $u>0$.  We set $X^{\psi_{0}}=0$.

 Consider an admissible sequence $(\mathcal{B}_{n})$ such that
   \begin{equation}
    \sum_{n\ge 0}2^{n}\Upsilon_{1}(B_{n}(\psi))\le 2\gamma_{1}(\Psi,d_{1}),\quad \forall \psi\in \Psi,
    \end{equation}
where $\Upsilon_{1}(B_{n}(\psi))$ is the diameter of the set $B_{n}(\psi)$ for $d_{1}$, and an admissible sequence $(\mathcal{C}_{n})$ such that
      \begin{equation}
      \sum_{n\ge 0}2^{n}\Upsilon_{2}(C_{n}(\psi))\le 2\gamma_{2}(\Psi,d_{2}),\quad \forall \psi\in \Psi,
      \end{equation}
where $\Upsilon_{2}(C_{n}(\psi))$ is the diameter of the set $C_{n}(\psi)$ for $d_{2}$.

We may define partition $\mathcal{A}_{n}$ for $\Psi$ as follows: $\mathcal{A}_{0}=\{\Psi\}$,
\begin{equation}
\mathcal{A}_{n}=\{B\cap C, B\in \mathcal{B}_{n-1},C\in \mathcal{C}_{n-1}\}.
\end{equation}
Consider a set $\Phi_{n}$ that contains exactly one point in  $\mathcal{A}_{n}$. For $\psi \in \Psi$, $\pi_n(\psi)$ is the element of $\Phi_{n}$
 that belong to $A_{n}(\psi)$.  We can easily obtain
     \begin{equation}
     X^{\psi}-X^{\psi_{0}}=\sum_{n\ge 1}(X^{\pi_{n}(\psi)}-X^{\pi_{n-1}(\psi)}).
     \end{equation}
Let $\Lambda_n$ be the event defined by
\begin{equation}
|X^{\pi_{n}(\psi)}-X^{\pi_{n-1}(\psi)}|\le u\big(2^{n}d_{1}(\pi_{n}(\psi),\pi_{n-1}(\psi))+2^{n/2}d_{2}(\pi_{n}(\psi),\pi_{n-1}(\psi))\big).
\end{equation}
For $u>1$,  it easily follows
  \begin{eqnarray*}
    \textsf{P}(\Lambda_n^c)\le \exp(-u 2^{n-1}).
  \end{eqnarray*}
Letting $\Omega_{u}=\cap_{n\ge 1}\Lambda_n$,  we obtain
          \begin{equation}
          \textsf{P}(\Omega_{u}^{c})\le p(u):=\sum_{n\ge 1}2\cdot 2^{2^{n+1}}\exp(-u 2^{n-1}).\end{equation}
On $\Omega_{u}$,
 \begin{equation}
 \big|X^{\psi}-X^{\psi_{0}}\big|\le u\sum_{n\ge 1}2^{n} d_{1}\big(\pi_{n}(\psi),\pi_{n-1}(\psi)\big)+2^{n/2}d_{2}\big(\pi_{n}(\psi),\pi_{n-1}(\psi)\big).
 \end{equation}
Hence,
       \begin{equation}
       \sup_{\psi \in \Psi}|X^{\psi}-X^{\psi_{0}}|\le uS,
       \end{equation}
where
        \begin{equation}
        S:=\sup_{\psi \in \Psi}\sum_{n\ge 1}2^{n}(d_{1}(\pi_{n}(\psi),\pi_{n-1}(\psi))+2^{n/2}d_{2}(\pi_{n}(\psi),\pi_{n-1}(\psi)).
        \end{equation}
Thus
      \begin{equation}
      \textsf{P}\Big(\sup_{\psi \in \Psi}|X^{\psi}-X^{\psi_{0}}|>uS\Big)\le p(u)
      \end{equation}
Obviously,
            \begin{equation}
            p(u)\le C\exp{(-\frac{u}{2})}.
            \end{equation}
When $n\ge 2$, we have $\pi_{n}(\psi)$, $\pi_{n-1}(\psi)\in A_{n-1}(\psi)\subset B_{n-2}(\psi) $, so that
       \begin{equation}
       d_{1}\big(\pi_{n}(\psi),\pi_{n-1}(\psi)\big)\le \Upsilon_{1}(B_{n-2}(\psi)).
       \end{equation}
Thus
   \begin{equation}
   \sum_{n\ge 1}2^{n}d_{1}(\pi_{n}(\psi))\le C\sum_{n\ge 1}2^{n}\Upsilon_{1}(B_{n-2}(\psi)) \le C \gamma_1(\Psi,d_{1})).
   \end{equation}

Proceeding similarly for $d_{2}$, we obtain
           \begin{equation}
           |X^{\psi}-X^{\psi_{0}}|\le Cu\big(\gamma_1(\Psi,d_{1})+\gamma_2(\Psi,d_{2})\big)
            \end{equation}

in $\Omega_{u}$. Thus
      \begin{equation}
      \textsf{P}\Big(\sup_{\psi \in \Psi}|X^{\psi}-X^{\psi_{0}}|>Cu\big(\gamma_1(\Psi,d_{1})+\gamma_2(\Psi,d_{2})\big)\Big)\le C\exp{(-\frac{u}{2})},
      \end{equation}
we complete the proof of  (\ref{ui}).
 
 We can obtain (\ref{mr}) through
    $$\textsf{E}|\xi|=\int_{0}^{+\infty}\textsf{P}(|\xi|>u)du. $$

\begin{remark}  Let   $(\Psi, d)$  be a metric space, and let $\big(X^\psi, \psi\in \Psi\big)$ be a family of stochastic processes  defined on a probability space $(\Omega, \mathcal{F}, \textsf{ P})$. A primary problem   is to study the bounds for $\textsf{E}\sup_{\psi\in \Psi}X^\psi$,  where
\begin{equation}
\textsf{E}\sup_{\psi\in \Psi}S^\psi=\sup\Big\{\textsf{E}\sup_{\psi\in \Phi}S^\psi; \Phi\subset \Psi, \Phi~\text{finite}\Big\}.
\end{equation}
However, this is not easy at all for general processes. The generic chaining method was first invented by Talagrand in a series of articles to deal with $\textsf{E}\sup_{\psi\in \Psi}X^\psi$. In particular,  under the increment condition
\begin{equation}
\textsf{P}(|X^{\psi_{1}}-X^{\psi_{2}}|>u)\le  2\exp\Big(-\frac{u^{2}}{2 d_{2}^{2}(\psi_{1},\psi_{2}) }\Big), \label{T0}
\end{equation}
Talagrand \cite{t} proved
\begin{equation*}
\textsf{P}\Big(\sup_{\psi \in \Psi}|X^{\psi}-X^{\psi_{0}}|>Cu\gamma_2(\Psi,d_{2})\Big)\le C\exp{(-\frac{u^2}{2})}. \label{TIN}
\end{equation*}
In addition, if  the condition (\ref{T0}) is replaced by
\begin{equation*}
\textsf{P}\big(|X^{\psi_{1}}-X^{\psi_{2}}|>u\big)\le  2\exp\Big(-\min\{\frac{u^{2}}{2 d_{2}^{2}(\psi_{1},\psi_{2}) },\frac{u}{2 d_{1}(\psi_{1},\psi_{2})}\}\Big), \label{TINC2}
\end{equation*}
then it follows
 \begin{equation}
 \textsf{P}\Big(\sup_{\psi \in \Psi}|X^{\psi}-X^{\psi_{0}}|>Cu(\gamma_1(\Psi,d_{1})+\gamma_2(\Psi,d_{2}))\Big)\le C\exp{(-\frac{u}{2})}.
 \end{equation}
Theorem \ref{mui} implies that if the following increment condition is satisfied
\begin{equation*}
\textsf{P}\big(|X^{\psi_{1}}-X^{\psi_{2}}|>u\big)\le  \exp\Big(-\frac{u^{2}}{2(d_{2}^{2}(\psi_{1},\psi_{2})+d_{1}(\psi_{1},\psi_{2})u)}\Big),
 \end{equation*}
then (\ref{ui}) still holds true.

\end{remark}

\section{Applications}

In this section we shall first apply the previous results to functional index empirical processes. Consider a sequence of adapted stationary time series $(Y_n)_{n\ge 0}$ on the discrete time  stochastic basis $(\Omega, \mathcal{F}, (\mathcal{F}_{n})_{n\ge 0},\textsf{ P})$.   Let $\Psi $ be the space of all bounded measurable functions in $\mathbb{R}$. For a $\psi\in \Psi$, define 
 \begin{equation}
 X_n^\psi= \sum_{k=1}^n\big(\psi(Y_k)-\textsf{E}(\psi(Y_k)|\mathcal{F}_{k-1})\big) .  \label{m}
 \end{equation}
Obviously, for each $\psi$, ${X_n^\psi}_{n\ge 0}$ is a discrete time martingale. Note also $X_n^\psi$ can be realized through a stochastic integral of $\psi$ with respect to a multivariate point process. In fact, let $T_k=k$, $X_k=Y_k$, $\psi(k, x)=\psi_(x_k)$, then 
\begin{equation}
 X_n^\psi=  \psi*(\mu-\nu)_n.
 \end{equation}
A simple computation shows
\begin{equation}
\psi*\nu(dt,dx)_n=\sum_{k=1}^{n}\textsf{E}[\psi(Y_k)|\mathcal{F}_{k-1}], 
 \end{equation}
 and
\begin{equation}
C(\psi)_n=\sum_{k=1}^{n}\textsf{E}[(\psi(Y_k))^2|\mathcal{F}_{k-1}].
\end{equation}
 As a direct consequence of  Theorem \ref{BESTIN}, we have
\begin{theorem}\label{ebe}
 Suppose that, for all $k>0$ and some $0< K<\infty$
 \begin{equation}\label{nc1}
 \textsf{E}\big[\max\{0, \psi(Y_k)\}^m|\mathcal{F}_{k-1}\big]\le \frac{m!}{2}K^{m-2}\textsf{E}\big[(\psi(Y_k))^2|\mathcal{F}_{k-1}\big]
 \end{equation}
 \begin{equation}\label{nc2}
 \textsf{E}\big[\max\{0, \psi(Y_k)\}|\mathcal{F}_{k-1}\big]\le K. 
  \end{equation}
 Then for each $x>0,$ $y>0$,
 \begin{equation}
 \textsf{P}(|X_n^\psi |\ge x~ \text{and}~C(\psi)_n\le y^{2}~\text{for some}~n)\le \exp\Big(-\frac{x^{2}}{2(xK+y^{2})}\Big).
 \end{equation}
\end{theorem}

\begin{remark}If we denote
\begin{equation}
\Xi(\psi)_{n}=\textsf{E}[\max\{0, \psi(Y_n)\}|\mathcal{F}_{n-1}],
\end{equation}
and
\begin{equation}
Q(\psi, m)_n=\sum_{k=1}^{n}\textsf{E}[\max\{0, \psi(Y_k)\}^m|\mathcal{F}_{k-1}],~~m\ge 3. 
\end{equation}
The conditions (\ref{nc1}) and (\ref{nc2})   imply the conditions (\ref{CW3}) in Theorem \ref{BESTIN} for $\Xi(\psi)_{n}$ and $Q(\psi, m)_n$.
\end{remark}
Furthermore, we define the metric  for fix $n>0$
\begin{equation}
d_{1}(\psi_1,\psi_2)=\max_{1\le k\le n}||\textsf{E}[\max\{0, \psi_1(Y_k)-\psi_2(Y_k)\}|\mathcal{F}_{k-1}]||_\infty,
\end{equation}
\begin{equation}
d_{2}(\psi_1,\psi_2)=\max_{1\le k\le n} ||\textsf{E}[ (\psi_1(Y_k)-\psi_2(Y_k))^2|\mathcal{F}_{k-1}]||^{1/2}_{\infty}.
\end{equation}
where $||\cdot||_\infty$ stands for norm of $L^{\infty}$.

\begin{theorem}\label{emui}
Suppose that, for all $\psi\in \Psi$ and some $0<K<\infty$, (\ref{nc1}) and (\ref{nc2}) hold. Then
\begin{equation}
\textsf{P} \Big(\sup_{\psi\in \Psi}\big|X_n^\psi\big|\ge Cu \big(\sqrt{n}\gamma_2(\Psi,d_{2})+n\gamma_1(\Psi,d_{1})\big)\Big)\le C\exp(-\frac{u}{2}). \label{eui}
\end{equation}
 
\end{theorem}

\begin{proof}

It follows directly from  the proof of Theorem \ref{mui} 
\begin{equation}
\textsf{P} \Big(\sup_{\psi\in \Psi}\big|X_n^\psi\big|\ge Cu \big(\gamma_2(\Psi, \sqrt{n}d_{2})+ \gamma_1(\Psi, nd_{1})\big)\Big)\le C\exp(-\frac{u}{2}).
\end{equation}
Note 
\begin{equation}
\gamma_2(\Psi,\sqrt{n}d_{2})=\sqrt{n}\gamma_2(\Psi,d_{2}), \quad \gamma_1(\Psi, nd_{1})=n\gamma_1(\Psi,d_{1}) 
\end{equation}
Then (\ref{eui}) easily holds.
 \end{proof}

As a special example of functional index empirical processes, we  consider  the   nonparametric maximum likelihood estimators below.

     Let $ \mathcal{P}=\{P_{\theta}, \theta\in \Theta\}$ be a family of probability measures, we assume that $ \mathcal{P}$ is dominated by a Lebesgue measure.  Denote the density of $P_{\theta}$ by $f_{\theta}=\frac{dP_{\theta}}{dx}$, $\theta\in \Theta$. Fix  a $\theta_0\in \Theta$ such that $f_{\theta_0}>0$,  and let $X_1, X_2, \cdots$ be   a sequence  of i.i.d. observations from $P_0=P_{\theta_0}$. Define the empirical distribution
   \begin{equation}
   P_n=\frac{1}{n}\sum_{k=1}^{n}\varepsilon_{(k, X_k)}(dt\times dx)
   \end{equation}
  on the basis of the first $n$ observations.  The maximum likelihood estimator $\hat{\theta}_n$ of $\theta_0$ is defined by
    \begin{equation}
    \int \log(f_{\hat{\theta}_n})d P_{n}=\max_{\theta \in \Theta} \int  \log(f_{\theta})d P_{n}. \label{Max}
    \end{equation}
    We assume throughout that a $\hat{\theta}_n$ exists.
    
 It is very important  to study the rate of convergence of $f_{\hat{\theta}_n}$ to $f_{\theta}$ in the theory of nonparametric statistical inference. Recall  the  Hellinger distance is usually used to describe the distance between two probability measures. In particular, for a pair $(P, \bar{P}) $ of probability measures  the Hellinger distance $H(P,\bar{P})$  is defined by
 \begin{equation}
 H^2(P,\bar{P})=\frac{1}{2}\int \Big(\sqrt{\frac{dP}{dQ}}-\sqrt{\frac{d\bar{P}}{dQ}}\Big)^2dQ
 \end{equation}
where $Q$ is a  measure dominating $P$ and $\bar{P}$.  

 In our setting $Q$ is a Lebesgue measure,  $f=\frac {dP}{dx}$, $\bar{f}=\frac {d\bar{P}}{dx}$, and we  simply write $h(f, \bar{f})=H(P, \bar{P})$.  It is natural to ask what the rate of  convergence for $f_{\hat{\theta}_n}$ to $f_{\theta_0}$  in terms of  $h^2(f_{\hat{\theta}_n}, f_{\theta_0})$ . We have the following result in this aspect.     Denote $\mathcal{G}=\{g_{\theta}:=\sqrt{\frac{f_{\theta}}{f_{\theta_{0}}}}-1,~~\theta\in \Theta\}$,  and set
\begin{equation}
d_1(g_{\theta_1},  g_{\theta_2} )=E_{\theta_0}|g_{\theta_1}(X_1)-g_{\theta_2}(X_1)|,
  \end{equation}
\begin{equation}
d_2(g_{\theta_1},  g_{\theta_2} )=\sqrt{E_{\theta_0}(g_{\theta_1}(X_1)-g_{\theta_2}(X_1))^2}.
 \end{equation}
  \begin{theorem}\label{muii}
Suppose that there is a positive constant $0<K<\infty$ such that for all $\theta\in \Theta$ 
 \begin{equation}
 E_{\theta_0}|g_{\theta}(X_1)|\le K,~~E_{\theta_0}|g_{\theta}(X_1)|^m\le \frac{m!}{2}K^{m-2}E_{\theta_0}|g_{\theta}(X_1)|^2,~~m\ge 3.
  \end{equation}
      Then it follows
\begin{equation}\label{uii}
\textsf{P} \big(h^2(f_{\hat{\theta}_n}, f_{\theta_0})\ge  u \big)\le C\exp\Big(-\frac{n u}{2C( \sqrt n\gamma_2(\mathcal{G},d_{2})+ \gamma_1(\mathcal{G},d_{1}))}\Big).
\end{equation}
 
\end{theorem}
   \begin{proof}

    Since $\log(1+x)\le x$ for any $x>-1$, 
     \begin{eqnarray*}
  0 &\le &  \frac{1}{2}\int \log \frac{f_{\theta}}{f_{\theta_0}} dP_n- \frac{1}{2}\int (f_{\theta}-f_{\theta_0})dx\\
   &\le&\int\Big(\sqrt{\frac{f_{\theta}}{f_{\theta_{0}}}}-1\Big) dP_n- \frac{1}{2}\int (f_{\theta}-f_{\theta_0})dx \\
   &=&\int\Big(\sqrt{\frac{f_{\theta}}{f_{\theta_{0}}}}-1\Big) d(P_n-P_0)+\int\Big(\sqrt{\frac{f_{\theta}}{f_{\theta_{0}}}}-1\Big) dP_0- \frac{1}{2}\int (f_{\theta}-f_{\theta_0})dx \\
&=&\int\Big(\sqrt{\frac{f_{\theta}}{f_{\theta_{0}}}}-1\Big) d(P_n-P_0)+\int\sqrt{f_{\theta}f_{\theta_{0}}} dx- \frac{1}{2}\int (f_{\theta}+f_{\theta_0})dx\\
&=&\int(\sqrt{\frac{f_{\theta}}{f_{\theta_{0}}}}-1) d(P_n-P_0)-h^2(f_{\theta}, f_{\theta_0}).
 \end{eqnarray*}
 Also, according to (\ref{Max}),
\begin{equation}
0\le  \frac{1}{2}\int \log(\frac{f_{\hat{\theta}_n}}{f_{\theta_0}})dP_n- \frac{1}{2}\int (f_{\hat{\theta}_n}-f_{\theta_0})dx.
\end{equation}
Thus we have
      \begin{eqnarray*}
      & &\textsf{P} \Big(h^2(f_{\hat{\theta}_n}, f_{\theta_0})\ge  Cu (\frac{1}{\sqrt{n}}\gamma_2(\mathcal{G},d_{2})+\frac{1}{n}\gamma_1(\mathcal{G},d_{1}))\Big)\\
      &\le&\textsf{P} \Big(\int(\sqrt{\frac{f_{\hat{\theta}_n}}{f_{\theta_{0}}}}-1) d(P_n-P_0)\ge  Cu (\frac{1}{\sqrt{n}}\gamma_2(\mathcal{G},d_{2})+\frac{1}{n}\gamma_1(\mathcal{G},d_{1}))\Big).
      \end{eqnarray*}
    We now can  complete the proof by Theorem \ref{emui}. 
\end{proof}

\begin{remark}
van de Geer \cite{v., v} discussed the similar problem on maximum likelihood estimators.  To our knowledge, Theorem \ref{muii} is  new in this area. We also remark that Theorem \ref{muii} can be extended  by Theorem    \ref{mui} and \ref{emui} to the stationary sample case with suitable maximum likelihood estimators. It is left to future work .
 \end{remark}

{\bf Acknowledgments}

This research work is support by the National Natural Science Foundation of China (No. 11371317, 11171303) and the Fundamental Research Fund of Shandong University (No. 2016GN019). 

\section*{Reference}

\bibliographystyle{amsplain}

 \end{document}